\documentclass[12pt]{amsart}
\usepackage{amssymb}
\usepackage[all]{xy}

\newtheorem{theorem}{Theorem}[section]
\newtheorem{definition}[theorem]{Definition}
\newtheorem{proposition}[theorem]{Proposition}
\newtheorem{conjecture}[theorem]{Conjecture}

\begin{document}

\title{Quantum groups from stationary matrix models}

\author{Teodor Banica}
\address{T.B.: Department of Mathematics, Cergy-Pontoise University, 95000 Cergy-Pontoise, France. {\tt teodor.banica@u-cergy.fr}}

\begin{abstract}
We study the quantum groups appearing via models $C(G)\subset M_K(C(X))$ which are ``stationary'', in the sense that the Haar integration over $G$ is the functional $tr\otimes\int_X$. Our results include a number of generalities, notably with a substantial list of examples, and a detailed discussion in the quantum permutation group case.
\end{abstract}


\maketitle

\section*{Introduction}

The compact quantum groups were introduced by Woronowicz in \cite{wo1}, \cite{wo2}, via a set of fairly simple and general axioms. Technically speaking, the theory is based on a good existence result for the Haar measure. To be more precise, the Haar integration can be constructed by starting with any faithful positive linear functional $\varphi\in C(G)^*$, and then by taking the Ces\`aro limit of the corresponding convolution powers:
$$\int_G=\lim_{k\to\infty}\frac{1}{k}\sum_{r=1}^k\varphi^{*r}$$

With this result in hand, a full generalization of the Peter-Weyl theory, and of Tannaka-Krein duality, can be developed, and this was done in \cite{wo1}, \cite{wo2}. Some simplifications, making the general theory even simpler, are now available from \cite{mvd}, \cite{mal}.

Generally speaking, the compact quantum groups are expected to be of help in quantum physics. The effects of quantum mechanics can be felt starting at the nanoscale level, and there is a wide array of situations where quantum groups of symmetries can appear. The potential fields of investigation include 2D lattice models, conformal field theory, quantum chromodynamics, and quantum information theory. However, the whole development here has been very slow, and all this remains to be confirmed. Looking back at the story of the subject, there are in fact two main reasons for this slow development:
\begin{enumerate}
\item Woronowicz's theory, while using mathematics from the 50s, was developed only in the 80s. Thus, we have here a 30-year gap, to be overcome.

\item The work so far in the area, accounting for another 30 years, has been mainly theoretical, or at least not applied enough.
\end{enumerate}

Summarizing, the compact quantum groups are a quite natural and beautiful theory, which remains however in need of a considerable amount of work.

In principle, the potential power of the compact quantum groups should come from their integration theory. Thanks to the efforts of many people, the knowledge here has substantially evolved since Woronowicz's original findings in \cite{wo1}, \cite{wo2}, and can be even labeled as reasonably advanced, among various noncommutative integration theories. There are two main ideas here, which have been systematically exploited:
\begin{enumerate}
\item Algebraic geometry. In analogy with the classical case, a Weingarten integration formula, based on Schur-Weyl duality, can be developed. This requires a good knowledge of the relations satisfied by the standard coordinates $u_{ij}\in C(G)$.

\item Random matrices. This is an alternative method, based on matrix models for the standard coordinates $u_{ij}$. Once a matrix model for $C(G)$ is found, with suitable faithfulness properties, the Haar integration can be recovered from it. 
\end{enumerate}

The aim of the present paper is to bring some new results in connection with (2) above. Given an inner faithful model $\pi:C(G)\to M_K(C(X))$, it is known from \cite{bfs}, \cite{wa2} that the Haar functional of $G$ can be recovered via a Ces\`aro limit, as follows:
$$\int_G=\lim_{k\to\infty}\frac{1}{k}\sum_{r=1}^k\left[\left(tr\otimes\int_X\right)\pi\right]^{*r}$$

We will study here the quantum groups appearing via models $C(G)\subset M_K(C(X))$ which are ``stationary'', in the sense that the convergence in the above Ces\`aro limiting result is stationary. This latter condition, which justifies the terminology, is in fact equivalent to the fact that the Haar integration is simply the random matrix trace: 
$$\int_G=tr\otimes\int_X$$

Observe that that our assumption $C(G)\subset M_K(C(X))$ implies that the dual quantum group $\Gamma=\widehat{G}$ must be amenable. This is of course quite restrictive. However, and here comes our main point, under this amenability assumption, the stationarity condition does not seem to be very restrictive. We have many examples, including:
\begin{enumerate}

\item The compact Lie groups, $G\subset U_N$. We can indeed take here $K=1$, and use the identity embedding $C(G)\subset M_1(C(G))$.

\item The finite quantum groups, $|G|<\infty$. Indeed, we can take here $K=|G|$, and use the left regular representation $\pi:C(G)\subset M_K(\mathbb C)$.

\item The half-classical orthogonal compact quantum groups, $G\subset O_N^*$. Here we can take $K=2$, and use the results in \cite{bdu}.

\item Certain classes of duals of amenable groups, $G=\widehat{\Gamma}$. Here the results are more technical, and will be explained in this paper.

\item Some key examples coming from the Pauli matrices, and their generalizations, the Weyl matrices, constructed in \cite{bco}, \cite{bne}.

\item A number of other examples in the quantum permutation group case, $G\subset S_N^+$, inspired from \cite{bne}, that we will discuss in this paper.
\end{enumerate}

In view of this list, there is some theory to be developed, with a lot of work to be done. Our results here, following the recent work in \cite{bne}, will include a number of generalities on the stationary models, then a detailed discussion of the above list of examples, notably with a proof of stationarity for the Weyl matrix models, and the discussion of some universal constructions, in the quantum permutation group case. 

As a conclusion, the present work aims to ``absorb'' large classes of quantum groups into random matrix theory. There are many questions here:
\begin{enumerate}
\item A first question regards the full understanding of the discrete group case. Here we have only very partial results, regarding the half-abelian case.

\item At the combinatorial and probabilistic level, there are some connections with the 2D lattice models (Ising, Potts). This remains to be clarified.

\item Regarding the non-stationary case, one question is that of unifying the present work with that in \cite{ban}, \cite{bb2}, \cite{bic}, on the deformed Fourier models.

\item Finally, one may wonder if all this can help in connection with the classification problem for the easy quantum groups \cite{fre}, \cite{rwe}, \cite{twe}. We do not know.
\end{enumerate}

The paper is organized as follows: 1-2 contain preliminaries and basic results, in 3-4 we discuss the models coming from the Weyl matrices, with a proof of their stationarity, and in 5-6 we study the universal flat models for $S_N^+$ and its subgroups.

\medskip

\noindent {\bf Acknowledgements.} I would like to thank Julien Bichon, Uwe Franz, Ion Nechita and Adam Skalski for several useful discussions. 

\section{Matrix models}

We will be interested in what follows in certain compact matrix quantum groups whose Haar functional appears as a random matrix trace. Since the traciality of the Haar functional corresponds to the Kac algebra assumption $S^2=id$, best is to use Woronowicz's compact quantum group formalism in \cite{wo1}, \cite{wo2}, with the extra axiom $S^2=id$. 

The precise definition that we will need is as follows:

\begin{definition}
Assume that $A$ is a $C^*$-algebra, and $u\in M_N(A)$ is a unitary matrix, such that the following formulae define morphisms of $C^*$-algebras:
$$\Delta(u_{ij})=\sum_ku_{ik}\otimes u_{kj}\quad,\quad\varepsilon(u_{ij})=\delta_{ij}\quad,\quad S(u_{ij})=u_{ji}^*$$
We write then $A=C(G)$, and call $G$ a compact matrix quantum group.
\end{definition}

The above maps $\Delta,\varepsilon,S$ are called comultiplication, counit and antipode. The basic examples include the compact Lie groups $G\subset U_N$, their $q$-deformations at $q=-1$, and the duals of the finitely generated discrete groups $\Gamma=<g_1,\ldots,g_N>$. See \cite{ntu}, \cite{wo1}.

Regarding now the matrix models, we use here:

\begin{definition}
A matrix model for $C(G)$ is a morphism of $C^*$-algebras 
$$\pi:C(G)\to M_K(C(X))$$
with $X$ being a compact space, and with $K\in\mathbb N$. 
\end{definition}

As a basic example, in the case where $G=\widehat{\Gamma}$ is the abstract dual of a discrete group, in the sense that we have $C(G)=C^*(\Gamma)$, such a matrix model $\pi:C^*(\Gamma)\to M_K(C(X))$ must come from a group representation $\rho:\Gamma\to C(X,U_K)$.

In the context of this latter example, observe that when $\rho$ is faithful, the induced representation $\pi$ is in general not faithful, its target algebra being finite dimensional. On the other hand, this representation ``reminds'' $\Gamma$. We say that $\pi$ is inner faithful.

We have in fact the following notions, coming from \cite{bb1}:

\begin{definition}
Let $\pi:C(G)\to M_K(C(X))$ be a matrix model. 
\begin{enumerate}
\item The Hopf image of $\pi$ is the smallest quotient Hopf $C^*$-algebra $C(G)\to C(H)$ producing a factorization of type $\pi:C(G)\to C(H)\to M_K(C(X))$.

\item When the inclusion $H\subset G$ is an isomorphism, i.e. when there is no non-trivial factorization as above, we say that $\pi$ is inner faithful.
\end{enumerate}
\end{definition}

Observe that when $G=\widehat{\Gamma}$ is a group dual, $\pi$ must come from a group representation $\rho:\Gamma\to C(X,U_K)$, and the above factorization is the one obtained by taking the image, $\rho:\Gamma\to\Gamma'\subset C(X,U_K)$. Thus $\pi$ is inner faithful when $\Gamma\subset C(X,U_K)$.

Also, given a compact group $G$, and elements $g_1,\ldots,g_K\in G$, we have a representation $\pi:C(G)\to\mathbb C^K$, given by $f\to(f(g_1),\ldots,f(g_K))$. The minimal factorization of $\pi$ is then via $C(G')$, with $G'=\overline{<g_1,\ldots,g_K>}$, and $\pi$ is inner faithful when $G=G'$.

We refer to \cite{bb1}, \cite{bcv}, \cite{chi} for more on these facts, and for a number of related algebraic results. In what follows, we will rather use analytic techniques. Assume indeed that $X$ is a probability space. We have then the following result, from \cite{bfs}, \cite{wa2}:

\begin{proposition}
Given an inner faithful model $\pi:C(G)\to M_K(C(X))$, we have
$$\int_G=\lim_{k\to\infty}\frac{1}{k}\sum_{r=1}^k\int_G^r$$
where $\int_G^r=(\varphi\circ\pi)^{*r}$, with $\varphi=tr\otimes\int_X$ being the random matrix trace.
\end{proposition}

\begin{proof}
This was proved in \cite{bfs} in the case $X=\{.\}$, using idempotent state theory from \cite{fsk}. The general case was recently established in \cite{wa2}. See \cite{sso}.
\end{proof}

The truncated integrals $\int_G^r$ can be evaluated as follows:

\begin{proposition}
Assuming that $\pi:C(G)\to M_K(C(X))$ maps $u_{ij}\to U_{ij}^x$, we have
$$\int_G^ru_{i_1j_1}^{e_1}\ldots u_{i_pj_p}^{e_p}=(T_e^r)_{i_1\ldots i_p,j_1\ldots j_p}$$
where $T_e\in M_{N^p}(\mathbb C)$ with $e\in\{1,*\}^p$ is given by $(T_e)_{i_1\ldots i_p,j_1\ldots j_p}=\int_Xtr(U_{i_1j_1}^{x,e_1}\ldots U_{i_pj_p}^{x,e_p})dx$.
\end{proposition}

\begin{proof}
This follows indeed from the definition of the various objects involved, namely from $\phi*\psi=(\phi\otimes\psi)\Delta$, and from $\Delta(u_{ij})=\sum_ku_{ik}\otimes u_{kj}$. See \cite{bb2}.
\end{proof}

We will be interested in what follows in the following notion:

\begin{definition}
A stationary model for $C(G)$ is a random matrix model
$$\pi:C(G)\to M_K(C(X))$$
having the property $\int_G=(tr\otimes\int_X)\pi$.
\end{definition}

As a first remark, any stationary model is faithful. Indeed, the stationarity condition gives a factorization $\pi:C(G)\to C(G)_{red}\subset M_K(C(X))$, and since the image algebra $C(G)_{red}$ follows to be of type I, and therefore nuclear, $G$ must be co-amenable, and so $\pi$ must be faithful. For some background on these questions, we refer to \cite{bmt}, \cite{ntu}.

We can study the stationary models by using the idempotent state technology from Proposition 1.4 and Proposition 1.5 above. The result here is as follows:

\begin{theorem}
For $\pi:C(G)\to M_K(C(X))$, the following are equivalent:
\begin{enumerate}
\item $Im(\pi)$ is a Hopf algebra, and $(tr\otimes\int_X)\pi$ is the Haar integration on it.

\item $\psi=(tr\otimes\int_X)\pi$ satisfies the idempotent state property $\psi*\psi=\psi$.

\item $T_e^2=T_e$, $\forall p\in\mathbb N$, $\forall e\in\{1,*\}^p$, where $(T_e)_{i_1\ldots i_p,j_1\ldots j_p}=(tr\otimes\int_X)(U_{i_1j_1}^{e_1}\ldots U_{i_pj_p}^{e_p})$.
\end{enumerate}
If these conditions are satisfied, we say that $\pi$ is stationary on its image.
\end{theorem}

\begin{proof}
Let us factorize our matrix model, as in Definition 1.3 above:
$$\pi:C(G)\to C(G')\to M_K(C(X))$$

Now observe that the conditions (1,2,3) only depend on the factorized representation $\pi':C(G')\to M_K(C(X))$. Thus, we can assume $G=G'$, which means that we can assume that $\pi$ is inner faithful. We can therefore use the formula in Proposition 1.4:
$$\int_G=\lim_{k\to\infty}\frac{1}{k}\sum_{r=1}^k\psi^{*r}$$

$(1)\implies(2)$ This is clear from definitions, because the Haar integration on any quantum group satisfies the equation $\psi*\psi=\psi$.

$(2)\implies(1)$ Assuming $\psi*\psi=\psi$, we have $\psi^{*r}=\psi$ for any $r\in\mathbb N$, and the above Ces\`aro limiting formula gives $\int_G=\psi$. By using now the amenability arguments explained after Definition 1.6, we obtain as well that $\pi$ is faithful, as desired.

In order to establish now $(2)\Longleftrightarrow(3)$, we use the formula in Proposition 1.5:
$$\psi^{*r}(u_{i_1j_1}^{e_1}\ldots u_{i_pj_p}^{e_p})=(T_e^r)_{i_1\ldots i_p,j_1\ldots j_p}$$

$(2)\implies(3)$ Assuming $\psi*\psi=\psi$, by using the above formula at $r=1,2$ we obtain that the matrices $T_e$ and $T_e^2$ have the same coefficients, and so they are equal.

$(3)\implies(2)$ Assuming $T_e^2=T_e$, by using the above formula at $r=1,2$ we obtain that the linear forms $\psi$ and $\psi*\psi$ coincide on any product of coefficients $u_{i_1j_1}^{e_1}\ldots u_{i_pj_p}^{e_p}$. Now since these coefficients span a dense subalgebra of $C(G)$, this gives the result.
\end{proof}

As a conclusion, from a random matrix viewpoint, the quantum groups having a stationary model are the ``simplest''. We will systematically investigate them in what follows, often by using the criterion in Theorem 1.7 (3) above, in order to detect them.

\section{Basic examples}

In this section we discuss some basic examples of compact quantum groups coming from stationary models. All our examples will have the feature that $X$ is in fact a compact Lie group, endowed with its Haar measure. So, let us begin with:

\begin{definition}
An algebraic stationary model for $C(G)$ is a model of type
$$\pi:C(G)\subset M_K(C(H))\quad,\quad\int_G=\left(tr\otimes\int_H\right)\pi$$
with $H$ being a compact Lie group, endowed with its Haar measure.
\end{definition}

There are of course many interesting stationary models which are not algebraic, but rather of ``universal'' nature. We will discussed them later on, in sections 4-5 below.

As a first trivial result, regarding the examples, we have:

\begin{proposition}
The following have algebraic stationary models:
\begin{enumerate}
\item The compact Lie groups.

\item The finite quantum groups.
\end{enumerate}
\end{proposition}

\begin{proof}
(1) This is clear, because we can use here the identity map:
$$id:C(G)\to M_1(C(G))$$

(2) This is clear as well, because we can use here the regular representation:
$$\lambda:C(G)\to M_{|G|}(\mathbb C)$$

To be more precise, if we endow the linear space $H=C(G)$ with the scalar product $<a,b>=\int_Ga^*b$, we have a representation $\lambda:C(G)\to B(H)$ given by $a\to[b\to ab]$. Now since we have $H\simeq\mathbb C^{|G|}$ with $|G|=\dim A$, we can view $\lambda$ as a matrix model map, as above, and the stationarity axiom $\int_G=tr\circ\lambda$ is satisfied, as desired. 
\end{proof}

In the group dual case, let us first recall that the matrix models $\pi:C^*(\Gamma)\to M_K(C(H))$ must come from group representations $\rho:\Gamma\to C(H,U_K)$, with $\pi$ being inner faithful if and only if $\rho$ is faithful. With this identification made, we have:

\begin{proposition}
An algebraic model $\rho:\Gamma\subset C(H,U_K)$ is stationary when:
$$\int_Htr(g^x)dx=0,\forall g\neq1$$
Moreover, the examples include all the abelian groups, and all finite groups.
\end{proposition}

\begin{proof}
Consider indeed a group embedding $\rho:\Gamma\subset C(H,U_K)$, which produces by linearity an inner faithful model $\pi:C^*(\Gamma)\to M_K(C(H))$. By linearity and continuity, it is enough to formulate the stationarity condition on the group elements $g\in C^*(\Gamma)$. With the notation $\rho(g)=(x\to g^x)$, this stationarity condition reads:
$$\int_Htr(g^x)dx=\delta_{g,1}$$

Since this equality is trivially satisfied at $g=1$, where by unitality of our representation we must have $g^x=1$ for any $x\in H$, we are led to the condition in the statement.

Regarding the examples, these are clear from Proposition 2.2. More precisely:

(1) When $\Gamma$ is abelian we can use the following trivial embedding:
$$\Gamma\subset C(\widehat{\Gamma},U_1)\quad:\quad g\to[\chi\to\chi(g)]$$

(2) When $\Gamma$ is finite we can use the left regular representation:
$$\Gamma\subset\mathcal L(\mathbb C\Gamma)\quad:\quad g\to[h\to gh]$$

Observe that in both cases, the stationarity condition is trivially satisfied.
\end{proof}

In general, deciding whether a given discrete group $\Gamma$ can have or not a stationary model looks like a quite subtle question. Note that, for this to hold, $\Gamma$ must be of course amenable. We will be back to these questions later on, on various occasions.

Let us recall now that the quantum group $O_N^*\subset O_N^+$ is constructed by imposing to the standard coordinates $u_{ij}\in C(O_N^+)$ the half-commutation relations $abc=cba$. Observe that we have $O_N\subset O_N^*\subset O_N^+$. We have the following result, coming from \cite{bdu}:

\begin{theorem}
Any half-classical quantum group $G\subset O_N^*$ has an algebraic stationary model, with $K=2$. 
\end{theorem}

\begin{proof}
In the classical case, $G\subset O_N$, we already know that we have such a model, with $K=1$. In the non-classical case now, $G\not\subset O_N$, it is known from \cite{bdu} that  we have a matrix model as follows, for a certain self-conjugate subgroup $H\subset U_N$:
$$\pi:C(G)\subset M_2(C(H))\quad:\quad u_{ij}=\begin{pmatrix}0&v_{ij}\\\bar{v}_{ij}&0\end{pmatrix}$$

Let us check now the stationarity condition, by using the criterion in Theorem 1.7 (3) above. Since the fundamental representation is self-adjoint, the various matrices $T_e$ with $e\in\{1,*\}^p$ are all equal. We denote this common matrix by $T_p$.

According to the definition of $T_p$, this matrix is given by:
$$(T_p)_{i_1\ldots i_p,j_1\ldots j_p}=\left(tr\otimes\int_H\right)\left[\begin{pmatrix}0&v_{i_1j_1}\\\bar{v}_{i_1j_1}&0\end{pmatrix}\ldots\ldots\begin{pmatrix}0&v_{i_pj_p}\\\bar{v}_{i_pj_p}&0\end{pmatrix}\right]$$

Since when multipliying an odd number of antidiagonal matrices we obtain an atidiagonal matrix, we have $T_p=0$ for $p$ odd. Also, when $p$ is even, we have:
\begin{eqnarray*}
(T_p)_{i_1\ldots i_p,j_1\ldots j_p}
&=&\left(tr\otimes\int_H\right)\begin{pmatrix}v_{i_1j_1}\ldots\bar{v}_{i_pj_p}&0\\0&\bar{v}_{i_1j_1}\ldots v_{i_pj_p}\end{pmatrix}\\
&=&\frac{1}{2}\left(\int_Hv_{i_1j_1}\ldots\bar{v}_{i_pj_p}+\int_H\bar{v}_{i_1j_1}\ldots v_{i_pj_p}\right)\\
&=&\int_HRe(v_{i_1j_1}\ldots\bar{v}_{i_pj_p})
\end{eqnarray*}

We have $T_p^2=T_p=0$ when $p$ is odd, so we are left with proving that we have $T_p^2=T_p$, when $p$ is even. For this purpose, we use the following formula:
$$Re(x)Re(y)=\frac{1}{2}\left(Re(xy)+Re(x\bar{y})\right)$$

By using this identity for each of the terms which appear in the product, and multi-index notations in order to simplify the writing, we obtain:
\begin{eqnarray*}
(T_p^2)_{ij}
&=&\sum_{k_1\ldots k_p}(T_p)_{i_1\ldots i_p,k_1\ldots k_p}(T_p)_{k_1\ldots k_p,j_1\ldots j_p}\\
&=&\int_H\int_H\sum_{k_1\ldots k_p}Re(v_{i_1k_1}\ldots\bar{v}_{i_pk_p})Re(w_{k_1j_1}\ldots\bar{w}_{k_pj_p})dvdw\\
&=&\frac{1}{2}\int_H\int_H\sum_{k_1\ldots k_p}Re(v_{i_1k_1}w_{k_1j_1}\ldots\bar{v}_{i_pk_p}\bar{w}_{k_pj_p})+Re(v_{i_1k_1}\bar{w}_{k_1j_1}\ldots\bar{v}_{i_pk_p}w_{k_pj_p})dvdw\\
&=&\frac{1}{2}\int_H\int_HRe((vw)_{i_1j_1}\ldots(\bar{v}\bar{w})_{i_pj_p})+Re((v\bar{w})_{i_1j_1}\ldots(\bar{v}w)_{i_pj_p})dvdw
\end{eqnarray*}

Now since $vw\in H$ is uniformly distributed when $v,w\in H$ are uniformly distributed, the quantity on the left integrates up to $(T_p)_{ij}$. Also, since $H$ is conjugation-stable, $\bar{w}\in H$ is uniformly distributed when $w\in H$ is uniformly distributed, so the quantity on the right integrates up to the same quantity, namely $(T_p)_{ij}$. Thus, we have:
$$(T_p^2)_{ij}=\frac{1}{2}\Big((T_p)_{ij}+(T_p)_{ij}\Big)=(T_p)_{ij}$$

Summarizing, we have obtained that for any $p$, the condition $T_p^2=T_p$ is satisfied. Thus Theorem 1.7 applies, and shows that our model is stationary, as claimed.
\end{proof}

As a consequence, let us work out the discrete group case:

\begin{proposition}
Any reflection group $\Gamma=<g_1,\ldots,g_N>$ which is ``half-abelian'', in the sense that $g_ig_jg_k=g_kg_jg_i$, has an algebraic stationary model, with $K=2$.
\end{proposition}

\begin{proof}
This follows from Theorem 2.4. To be more precise, in the non-abelian case, the results in \cite{bdu} show that $\widehat{\Gamma}\subset O_N^*$ must come from a group dual $\widehat{\Lambda}\subset U_N$, via the construction there, and with $\Lambda=<h_1,\ldots,h_N>$, the corresponding model is:
$$\Gamma\subset C(\widehat{\Lambda},U_2)\quad:\quad g_i\to \left[\chi\to\begin{pmatrix}0&\chi(h_i)\\ \bar{\chi}(h_i)&0\end{pmatrix}\right]$$

As for the abelian case, the result here follows from Proposition 2.3 above. 
\end{proof}

Observe that Theorem 2.4 fully extends Proposition 2.2 (1), in the orthogonal case. In the general unitary case, the recent results in \cite{bb3} suggest that any closed quantum subgroup $G\subset U_N^{**}$ should have an algebraic stationary $2\times2$ model, constructed by using certain antidiagonal matrices. We intend to discuss this in a future paper.

\section{Weyl models}

We discuss now some more subtle examples of algebraic models, coming from the Weyl matrices. Our starting point is the following key definition, due to Wang \cite{wa1}:

\begin{definition}
$C(S_N^+)$ is the universal algebra generated by the entries of a $N\times N$ matrix $w=(w_{ij})$ which is magic, in the sense that its entries are projections ($p=p^2=p^*$), summing up to $1$ on each row and each column.
\end{definition}

This algebra satisfies Woronowicz' axioms in \cite{wo1}, \cite{wo2}, and the underlying space $S_N^+$ is therefore a compact quantum group, called quantum permutation group.

Observe that any magic unitary $u\in M_N(A)$ produces a representation $\pi:C(S_N^+)\to A$, given by $\pi(w_{ij})=u_{ij}$. In particular, we have a representation as follows:
$$\pi:C(S_N^+)\to C(S_N)\quad:\quad w_{ij}\to\chi\left(\sigma\in S_N\big|\sigma(j)=i\right)$$

The corresponding embedding $S_N\subset S_N^+$ is an isomorphism at $N=2,3$, but not at $N\geq4$, where $S_N^+$ is a non-classical, infinite compact quantum group. See \cite{wa1}.

A key result, going back to \cite{bco}, was the construction of a stationary model for $C(S_4^+)$, using the Pauli matrices. We will review now this result, along with the recent generalizations from \cite{bne}. Our starting point is the following definition:

\begin{definition}
Given a finite abelian group $H$, the associated Weyl matrices are
$$W_{ia}:e_b\to<i,b>e_{a+b}$$
where $i\in H$, $a,b\in\widehat{H}$, and where $(i,b)\to<i,b>$ is the Fourier coupling $H\times\widehat{H}\to\mathbb T$.
\end{definition}

As a basic example, consider the cyclic group $H=\mathbb Z_2=\{0,1\}$. Here the Fourier coupling is given by $<i,b>=(-1)^{ib}$, and so the Weyl matrices act via $W_{00}:e_b\to e_b$, $W_{10}:e_b\to(-1)^be_b$, $W_{11}:e_b\to(-1)^be_{b+1}$, $W_{01}:e_b\to e_{b+1}$. Thus, we have:
$$W_{00}=\begin{pmatrix}1&0\\0&1\end{pmatrix}\ ,\ 
W_{10}=\begin{pmatrix}1&0\\0&-1\end{pmatrix}\ ,\ 
W_{11}=\begin{pmatrix}0&-1\\1&0\end{pmatrix}\ ,\ 
W_{01}=\begin{pmatrix}0&1\\1&0\end{pmatrix}$$

We recognize here, up to some multiplicative factors, the four Pauli matrices.

Now back to the general case, we have the following well-known result:

\begin{proposition}
The Weyl matrices are unitaries, and satisfy:
\begin{enumerate}
\item $W_{ia}^*=<i,a>W_{-i,-a}$.

\item $W_{ia}W_{jb}=<i,b>W_{i+j,a+b}$.

\item $W_{ia}W_{jb}^*=<j-i,b>W_{i-j,a-b}$.

\item $W_{ia}^*W_{jb}=<i,a-b>W_{j-i,b-a}$.
\end{enumerate}
\end{proposition}

\begin{proof}
The unitary follows from (3,4), and the rest of the proof goes as follows:

(1) We have indeed the following computation:
\begin{eqnarray*}
W_{ia}^*
&=&\left(\sum_b<i,b>E_{a+b,b}\right)^*
=\sum_b<-i,b>E_{b,a+b}\\
&=&\sum_b<-i,b-a>E_{b-a,b}
=<i,a>W_{-i,-a}
\end{eqnarray*}

(2) Here the verification goes as follows:
\begin{eqnarray*}
W_{ia}W_{jb}
&=&\left(\sum_d<i,b+d>E_{a+b+d,b+d}\right)\left(\sum_d<j,d>E_{b+d,d}\right)\\
&=&\sum_d<i,b><i+j,d>E_{a+b+d,d}=<i,b>W_{i+j,a+b}
\end{eqnarray*}

(3,4) By combining the above two formulae, we obtain:
\begin{eqnarray*}
W_{ia}W_{jb}^*&=&<j,b>W_{ia}W_{-j,-b}=<j,b><i,-b>W_{i-j,a-b}\\
W_{ia}^*W_{jb}&=&<i,a>W_{-i,-a}W_{jb}=<i,a><-i,b>W_{j-i,b-a}
\end{eqnarray*}

But this gives the formulae in the statement, and we are done.
\end{proof}

Observe that, with $n=|H|$, we can use an isomorphism $l^2(\widehat{H})\simeq\mathbb C^n$ as to view each $W_{ia}$ as a usual matrix, $W_{ia}\in M_n(\mathbb C)$, and hence as a usual unitary, $W_{ia}\in U_n$.

Given a vector $\xi$, we denote by $Proj(\xi)$ the orthogonal projection onto $\mathbb C\xi$.

Now let $N=n^2$, and consider Wang's quantum permutation algebra $C(S_N^+)$, with standard generators denoted $w_{ia,jb}$, using double indices. Following \cite{bne}, we have:

\begin{theorem}
Given a closed subgroup $E\subset U_n$, we have a representation
$$\pi_H:C(S_N^+)\to M_N(C(E))\quad:\quad w_{ia,jb}\to[U\to Proj(W_{ia}UW_{jb}^*)]$$
where $n=|H|,N=n^2$, and where $W_{ia}$ are the Weyl matrices associated to $H$.
\end{theorem}

\begin{proof}
The Weyl matrices being given by $W_{ia}:e_b\to<i,b>e_{a+b}$, we have:
$$tr(W_{ia})=\begin{cases}
1&{\rm if}\ (i,a)=(0,0)\\
0&{\rm if}\ (i,a)\neq(0,0)
\end{cases}$$

Together with the formulae in Proposition 3.3, this shows that the Weyl matrices are pairwise orthogonal with respect to the scalar product $<x,y>=tr(x^*y)$ on $M_n(\mathbb C)$. Thus, these matrices form altogether an orthogonal basis of $M_n(\mathbb C)$, consisting of unitaries:
$$W=\left\{W_{ia}\Big|i\in H,a\in\widehat{H}\right\}$$

Thus, each row and each column of the matrix $\xi_{ia,jb}=W_{ia}UW_{jb}^*$ is an orthogonal basis of $M_n(\mathbb C)$, and so the corresponding projections form a magic unitary, as claimed.
\end{proof}

The above models, with $E=U_n$, were introduced and studied in \cite{bco} for the group $H=\mathbb Z_2$, where the Weyl matrices are the Pauli matrices, and in \cite{bne} in general. 

The main result in \cite{bco} was the stationarity of the model. We will generalize here this fact, to the case where $H$ is arbitrary, and where $W\subset E\subset U_n$ is arbitrary as well.

As for the main result in \cite{bne}, this was the computation of the law of the main character, for $E=U_n$, which turned out to be the same one as for $PU_n=U_n/\mathbb T$. As explained there, this suggests that the image of the model should be a twist, $C(PU_n)^\sigma$. We have no advances here, but our results below suggest that such a result should be valid for an arbitary compact group $W\subset E\subset U_n$, with the image being conjecturally $C(PE)^\sigma$.

\section{Stationarity}

In order to investigate stationarity questions for the Weyl matrix models, the idea will be to first compute the matrices $T_p=T_e$, and then to prove that we have $T_p^2=T_p$.

We recall that our scalar products, usually of type $<x,y>=tr(x^*y)$, are by definition linear at right. With this convention, we have the following well-known result:

\begin{proposition}
With $T=Proj(x_1)\ldots Proj(x_p)$ and $||x_i||=1$ we have 
$$<\xi,T\eta>=<\xi,x_1><x_1,x_2>\ldots<x_{p-1},x_p><x_p,\eta>$$
for any $\xi,\eta$. In particular, $Tr(T)=<x_1,x_2><x_2,x_3>\ldots<x_p,x_1>$.
\end{proposition}

\begin{proof}
For $||x||=1$ we have $Proj(x)\eta=x<x,\eta>$, and this gives:
\begin{eqnarray*}
T\eta
&=&Proj(x_1)\ldots Proj(x_p)\eta\\
&=&Proj(x_1)\ldots Proj(x_{p-1})x_p<x_p,\eta>\\
&=&Proj(x_1)\ldots Proj(x_{p-2})x_{p-2}<x_{p-1},x_p><x_p,\eta>\\
&=&\ldots\\
&=&x_1<x_1,x_2>\ldots<x_{p-1},x_p><x_p,\eta>
\end{eqnarray*}

Now by taking the scalar product with $\xi$, this gives the first assertion. As for the second assertion, this follows from the first assertion, by summing over $\xi=\eta=e_i$.
\end{proof}

Now back to the Weyl matrix models, let us first compute $T_p$. We have:

\begin{proposition}
We have the formula
\begin{eqnarray*}
(T_p)_{ia,jb}
&=&\frac{1}{N}<i_1,a_1-a_2>\ldots<i_p,a_p-a_1><j_2,b_2-b_1>\ldots<j_1,b_1-b_p>\\
&&\int_Etr(W_{i_2-i_1,a_2-a_1}UW_{j_1-j_2,b_1-b_2}U^*)\ldots tr(W_{i_1-i_p,a_1-a_p}UW_{j_p-j_1,b_p-b_1}U^*)dU
\end{eqnarray*}
with all the indices varying in a cyclic way.
\end{proposition}

\begin{proof}
By using the trace formula in Proposition 4.1 above, we obtain:
\begin{eqnarray*}
(T_p)_{ia,jb}
&=&\left(tr\otimes\int_E\right)\left(Proj(W_{i_1a_1}UW_{j_1b_1}^*)\ldots Proj(W_{i_pa_p}UW_{j_pb_p}^*)\right)\\
&=&\frac{1}{N}\int_E<W_{i_1a_1}UW_{j_1b_1}^*,W_{i_2a_2}UW_{j_2b_2}^*>\ldots<W_{i_pa_p}UW_{j_pb_p}^*,W_{i_1a_1}UW_{j_1b_1}^*>dU
\end{eqnarray*}

In order to compute now the scalar products, observe that we have:
\begin{eqnarray*}
<W_{ia}UW_{jb}^*,W_{kc}UW_{ld}^*>
&=&tr(W_{jb}U^*W_{ia}^*W_{kc}UW_{ld}^*)\\
&=&tr(W_{ia}^*W_{kc}UW_{ld}^*W_{jb}U^*)\\
&=&<i,a-c><l,d-b>tr(W_{k-i,c-a}UW_{j-l,b-d}U^*)
\end{eqnarray*}

By plugging these quantities into the formula of $T_p$, we obtain the result.
\end{proof}

Consider now the Weyl group $W=\{W_{ia}\}\subset U_n$, that we already met in the proof of Theorem 3.4 above. We have the following result:

\begin{theorem}
For any compact group $W\subset E\subset U_n$, the model
$$\pi_H:C(S_N^+)\to M_N(C(E))\quad:\quad w_{ia,jb}\to[U\to Proj(W_{ia}UW_{jb}^*)]$$
constructed above is stationary on its image, in the sense of Theorem 1.7.
\end{theorem}

\begin{proof}
We must prove that we have $T_p^2=T_p$. We have:
\begin{eqnarray*}
(T_p^2)_{ia,jb}
&=&\sum_{kc}(T_p)_{ia,kc}(T_p)_{kc,jb}\\
&=&\frac{1}{N^2}\sum_{kc}<i_1,a_1-a_2>\ldots<i_p,a_p-a_1><k_2,c_2-c_1>\ldots<k_1,c_1-c_p>\\
&&<k_1,c_1-c_2>\ldots<k_p,c_p-c_1><j_2,b_2-b_1>\ldots<j_1,b_1-b_p>\\
&&\int_Etr(W_{i_2-i_1,a_2-a_1}UW_{k_1-k_2,c_1-c_2}U^*)\ldots tr(W_{i_1-i_p,a_1-a_p}UW_{k_p-k_1,c_p-c_1}U^*)dU\\
&&\int_Etr(W_{k_2-k_1,c_2-c_1}VW_{j_1-j_2,b_1-b_2}V^*)\ldots tr(W_{k_1-k_p,c_1-c_p}VW_{j_p-j_1,b_p-b_1}V^*)dV
\end{eqnarray*}

By rearranging the terms, this formula becomes:
\begin{eqnarray*}
(T_p^2)_{ia,jb}
&=&\frac{1}{N^2}<i_1,a_1-a_2>\ldots<i_p,a_p-a_1><j_2,b_2-b_1>\ldots<j_1,b_1-b_p>\\
&&\int_E\int_E\sum_{kc}<k_1-k_2,c_1-c_2>\ldots<k_p-k_1,c_p-c_1>\\
&&tr(W_{i_2-i_1,a_2-a_1}UW_{k_1-k_2,c_1-c_2}U^*)tr(W_{k_2-k_1,c_2-c_1}VW_{j_1-j_2,b_1-b_2}V^*)\\
&&\hskip50mm\ldots\ldots\\
&&tr(W_{i_1-i_p,a_1-a_p}UW_{k_p-k_1,c_p-c_1}U^*)tr(W_{k_1-k_p,c_1-c_p}VW_{j_p-j_1,b_p-b_1}V^*)dUdV
\end{eqnarray*}

Let us denote by $I$ the above double integral. By using $W_{kc}^*=<k,c>W_{-k,-c}$ for each of the couplings, and by moving as well all the $U^*$ variables to the left, we obtain:
\begin{eqnarray*}
I
&=&\int_E\int_E\sum_{kc}tr(U^*W_{i_2-i_1,a_2-a_1}UW_{k_1-k_2,c_1-c_2})tr(W_{k_1-k_2,c_1-c_2}^*VW_{j_1-j_2,b_1-b_2}V^*)\\
&&\hskip50mm\ldots\ldots\\
&&tr(U^*W_{i_1-i_p,a_1-a_p}UW_{k_p-k_1,c_p-c_1})tr(W_{k_p-k_1,c_p-c_1}^*VW_{j_p-j_1,b_p-b_1}V^*)dUdV
\end{eqnarray*}

In order to perform now the sums, we use the following formula:
\begin{eqnarray*}
tr(AW_{kc})tr(W_{kc}^*B)
&=&\frac{1}{N}\sum_{qrst}A_{qr}(W_{kc})_{rq}(W^*_{kc})_{st}B_{ts}\\
&=&\frac{1}{N}\sum_{qrst}A_{qr}<k,q>\delta_{r-q,c}<k,-s>\delta_{t-s,c}B_{ts}\\
&=&\frac{1}{N}\sum_{qs}<k,q-s>A_{q,q+c}B_{s+c,s}
\end{eqnarray*}

If we denote by $A_x,B_x$ the variables which appear in the formula of $I$, we have:
\begin{eqnarray*}
I
&=&\frac{1}{N^p}\int_E\int_E\sum_{kcqs}<k_1-k_2,q_1-s_1>\ldots<k_p-k_1,q_p-s_p>\\
&&(A_1)_{q_1,q_1+c_1-c_2}(B_1)_{s_1+c_1-c_2,s_1}\ldots (A_p)_{q_p,q_p+c_p-c_1}(B_p)_{s_p+c_p-c_1,s_p}\\
&=&\frac{1}{N^p}\int_E\int_E\sum_{kcqs}<k_1,q_1-s_1-q_p+s_p>\ldots<k_p,q_p-s_p-q_{p-1}+s_{p-1}>\\
&&(A_1)_{q_1,q_1+c_1-c_2}(B_1)_{s_1+c_1-c_2,s_1}\ldots (A_p)_{q_p,q_p+c_p-c_1}(B_p)_{s_p+c_p-c_1,s_p}
\end{eqnarray*}

Now observe that we can perform the sums over $k_1,\ldots,k_p$. We obtain in this way a multiplicative factor $n^p$, along with the condition $q_1-s_1=\ldots=q_p-s_p$. Thus we must have $q_x=s_x+a$ for a certain $a$, and the above formula becomes:
$$I=\frac{1}{n^p}\int_E\int_E\sum_{csa}(A_1)_{s_1+a,s_1+c_1-c_2+a}(B_1)_{s_1+c_1-c_2,s_1}\ldots(A_p)_{s_p+a,s_p+c_p-c_1+a}(B_p)_{s_p+c_p-c_1,s_p}$$

Consider now the variables $r_x=c_x-c_{x+1}$, which altogether range over the set $Z$ of multi-indices having sum 0. By replacing the sum over $c_x$ with the sum over $r_x$, which creates a multiplicative $n$ factor, we obtain the following formula:
$$I=\frac{1}{n^{p-1}}\int_E\int_E\sum_{r\in Z}\sum_{sa}(A_1)_{s_1+a,s_1+r_1+a}(B_1)_{s_1+r_1,s_1}\ldots(A_p)_{s_p+a,s_p+r_p+a}(B_p)_{s_p+r_p,s_p}$$

Since for an arbitrary multi-index $r$ we have $\delta_{\sum_ir_i,0}=\frac{1}{n}\sum_i<i,r_1>\ldots<i,r_p>$, we can replace the sum over $r\in Z$ by a full sum, as follows:
\begin{eqnarray*}
I
&=&\frac{1}{n^p}\int_E\int_E\sum_{rsia}<i,r_1>(A_1)_{s_1+a,s_1+r_1+a}(B_1)_{s_1+r_1,s_1}\\
&&\hskip40mm\ldots\ldots\\
&&\hskip20mm<i,r_p>(A_p)_{s_p+a,s_p+r_p+a}(B_p)_{s_p+r_p,s_p}
\end{eqnarray*}

In order to ``absorb'' now the indices $i,a$, we can use the following formula:
\begin{eqnarray*}
W_{ia}^*AW_{ia}
&=&\left(\sum_b<i,-b>E_{b,a+b}\right)\left(\sum_{bc}E_{a+b,a+c}A_{a+b,a+c}\right)\left(\sum_c<i,c>E_{a+c,c}\right)\\
&=&\sum_{bc}<i,c-b>E_{bc}A_{a+b,a+c}
\end{eqnarray*}

Thus we have $(W_{ia}^*AW_{ia})_{bc}=<i,c-b>A_{a+b,a+c}$, and our formula becomes:
\begin{eqnarray*}
I
&=&\frac{1}{n^p}\int_E\int_E\sum_{rsia}(W_{ia}^*A_1W_{ia})_{s_1,s_1+r_1}(B_1)_{s_1+r_1,s_1}\ldots(W_{ia}^*A_pW_{ia})_{s_p,s_p+r_p}(B_p)_{s_p+r_p,s_p}\\
&=&\int_E\int_E\sum_{ia}tr(W_{ia}^*A_1W_{ia}B_1)\ldots\ldots tr(W_{ia}^*A_pW_{ia}B_p)
\end{eqnarray*}

Now by replacing $A_x,B_x$ with their respective values, we obtain:
\begin{eqnarray*}
I
&=&\int_E\int_E\sum_{ia}tr(W_{ia}^*U^*W_{i_2-i_1,a_2-a_1}UW_{ia}VW_{j_1-j_2,b_1-b_2}V^*)\\
&&\hskip30mm\ldots\ldots\\
&&tr(W_{ia}^*U^*W_{i_1-i_p,a_1-a_p}UW_{ia}VW_{j_p-j_1,b_p-b_1}V^*)dUdV
\end{eqnarray*}

By moving the $W_{ia}^*U^*$ variables at right, we obtain, with $S_{ia}=UW_{ia}V$:
\begin{eqnarray*}
I
&=&\sum_{ia}\int_E\int_Etr(W_{i_2-i_1,a_2-a_1}S_{ia}W_{j_1-j_2,b_1-b_2}S_{ia}^*)\\
&&\hskip30mm\ldots\ldots\\
&&tr(W_{i_1-i_p,a_1-a_p}S_{ia}W_{j_p-j_1,b_p-b_1}S_{ia}^*)dUdV
\end{eqnarray*}

Now since $S_{ia}$ is Haar distributed when $U,V$ are Haar distributed, we obtain:
$$I=N\int_E\int_Etr(W_{i_2-i_1,a_2-a_1}UW_{j_1-j_2,b_1-b_2}U^*)\ldots tr(W_{i_1-i_p,a_1-a_p}UW_{j_p-j_1,b_p-b_1}U^*)dU$$

But this is exactly $N$ times the integral in the formula of $(T_p)_{ia,jb}$, from Proposition 4.2 above. Since the $N$ factor cancels with one of the two $N$ factors that we found in the beginning of the proof, when first computing $(T_p^2)_{ia,jb}$, we are done.
\end{proof}

As explained at the end of the previous section, the above result raises a number of interesting algebraic questions. Analytically speaking now, a main problem is that of converting the formula of $T_p$ in Proposition 4.2 into a Weingarten type formula, by using the Fourier transform tricks in \cite{bne}, and the classical Weingarten formula \cite{csn}, \cite{wei}.

Finally, at both the combinatorial and probabilistic level, the results from \cite{bne} and from here have a certain similarity with the theory of certain statistical mechanical models, of Ising and Potts type \cite{bax}, which is waiting as well to be understood.

\section{Universal models}

We discuss in what follows a number of universal ``rank 1'' matrix model constructions for the quantum permutation groups, following \cite{bne}. First, we have:

\begin{definition}
A flat magic unitary is a magic unitary of the form
$$u\in M_N(M_N(\mathbb C))$$
with each $u_{ij}\in M_N(\mathbb C)$ being a rank $1$ projection.
\end{definition}

As a basic example, the Weyl matrix models come from flat magic unitaries. The terminology comes from the fact that the matrix $d_{ij}=tr(u_{ij})$, which is bistochastic, with sum 1 on each row and each column, must be the ``flat'' matrix, $d=(\frac{1}{N})_{ij}$.

In general, if we write $u_{ij}=Proj(\xi_{ij})$, with $||\xi_{ij}||=1$, uniquely determined up to parameters $\tau_{ij}\in\mathbb T$, then the array of vectors $\xi=(\xi_{ij})$ is a ``magic basis'', in the sense that each of its rows and  columns is an orthonormal basis of $\mathbb C^N$. Conversely, given a magic basis $\xi=(\xi_{ij})$, the projections $u_{ij}=Proj(\xi_{ij})$ form a flat magic unitary.

These basic facts are best viewed in the following way:

\begin{definition}
Associated to any $N\in\mathbb N$ are compact spaces $X_N,\widetilde{X}_N$ as follows,
$$\begin{matrix}
\widetilde{X}_N&\subset&M_N(S^{N-1}_\mathbb C)\\
\\
\downarrow&&\downarrow\\
\\
X_N&\subset&M_N(P^{N-1}_\mathbb C)
\end{matrix}$$
consisting of all $N\times N$ flat magic unitaries $(u_{ij})$, and all $N\times N$ magic bases $(\xi_{ij})$.
\end{definition}

Here $S^{N-1}_\mathbb C\subset\mathbb C^N$ is the unit sphere, and $P^{N-1}_\mathbb C=S^{N-1}_\mathbb C/\mathbb T$ is the corresponding projective space, whose elements can be identified with the rank 1 projections in $M_N(\mathbb C)$. The horizontal inclusions come from definitions, the map on the left is given by $\xi\to u$ with $u_{ij}=Proj(\xi_{ij})$, and the map on the right is induced by the canonical quotient map $S^{N-1}_\mathbb C\to P^{N-1}_\mathbb C$. Observe that $\widetilde{X}_N$ appears as lift of $X_N$, via the map on the right.

Now back to our matrix model problematics, we have:

\begin{definition}
The representation
$$\pi_N:C(S_N^+)\to M_N(C(X_N))$$
given by $w_{ij}\to[u\to u_{ij}]$ is called universal flat model of $C(S_N^+)$. 
\end{definition}

Observe that the above matrix model is indeed universal, among the flat matrix models for $C(S_N^+)$. This is indeed a trivial statement, which follows from definitions.

We will construct now a universal flat model for any closed subgroup $G\subset S_N^+$. Let us first discuss the case of the usual permutation group, $S_N\subset S_N^+$. We agree to denote by $\sim$ the proportionality of vectors, and by $\perp$, their orthogonality. We have:

\begin{definition}
Associated to any $N\in\mathbb N$ are spaces $X_N^\circ,\widetilde{X}^\circ_N$ as follows,
$$\begin{matrix}
\widetilde{X}^\circ_N&\subset&\widetilde{X}_N&\subset&M_N(S^{N-1}_\mathbb C)\\
\\
\downarrow&&\downarrow&&\downarrow\\
\\
X_N^\circ&\subset&X_N&\subset&M_N(P^{N-1}_\mathbb C)
\end{matrix}$$
consisting of all $N\times N$ flat magic unitaries $(u_{ij})$ with commuting entries, and all $N\times N$ magic bases $(\xi_{ij})$ having the property $[\xi_{ij}\sim\xi_{kl}$ or $\xi_{ij}\perp\xi_{kl}]$, for any $i,j,k,l$.
\end{definition}

Observe that we have indeed maps as above, the main observation here being the fact that $Proj(\xi),Proj(\eta)$ commute precisely when $\xi,\eta$ are proportional, or orthogonal.

The above spaces have in fact a very simple structure. Consider the following spaces, with the elements of $P^{N-1}_\mathbb C$ being regarded as usual as rank one projections:
\begin{eqnarray*}
(S^{N-1}_\mathbb C)^{N,\perp}&=&\left\{(x_1,\ldots,x_N)\in (S^{N-1}_\mathbb C)^N\Big|x_i\perp x_j,\forall i\neq j\right\}\\ 
(P^{N-1}_\mathbb C)^{N,\perp}&=&\left\{(u_1,\ldots,u_N)\in (P^{N-1}_\mathbb C)^N\Big|u_iu_j=0,\forall i\neq j\right\}
\end{eqnarray*}

Let us recall as well that a Latin square is a square matrix $L\in M_N(1,\ldots,N)$, all whose rows and columns are permutations of $1,\ldots,N$. We call $L$ half-normalized if its first row is $1,\ldots,N$, and normalized is both its first row and first column are $1,\ldots,N$.

Finally, consider the universal flat representation of $C(S_N)$, namely:
$$\pi_N^\circ:C(S_N)\to M_N(C(X_N^\circ))\quad:\quad w_{ij}\to[u\to u_{ij}]$$

With these conventions, we have the following result:

\begin{theorem}
We have identifications as follows,
$$X_N^\circ=(P^{N-1}_\mathbb C)^{N,\perp}\times L_N\quad,\quad \widetilde{X}_N^\circ=(S^{N-1}_\mathbb C)^{N,\perp}\times L_N\times\mathbb T^{N(N-1)}$$
where $L_N$ is the set of half-normalized $N\times N$ Latin squares, and $\pi_N^\circ$ is stationary.
\end{theorem}

\begin{proof}
Consider a flat magic unitary $u\in X_N$, and let $(u_1,\ldots,u_N)\in(P^{N-1}_\mathbb C)^{N,\perp}$ be its first row. Since the condition $u\in X_N^\circ$ tells us that we must have $u_{ij}=u_{L_{ij}}$, for a certain half-normalized Latin square $L$, this gives the first assertion. 

The second assertion, regarding $\widetilde{X}_N^\circ$, follows by taking the affine lift, with $\mathbb T^{N(N-1)}$ being the space of phases for the entries in the remaining rows.

Finally, let us endow $X_N^\circ$ with the homogeneous space measure on $(P^{N-1}_\mathbb C)^{N,\perp}$ times the counting measure on $L_N$. By invariance, it is enough to fix an arbitrary point $u\in (P^{N-1}_\mathbb C)^{N,\perp}$, and prove that the corresponding fiber $\pi^u:C(S_N)\to M_N(C(L_N))$ is stationary. But this follows from the fact that the symmetric group acts on $L_N$, and so the random matrix trace on $C(S_N)$ is invariant under permutations.
\end{proof}

\section{Quantum subgroups}

In this section we discuss the general case, where $G\subset S_N^+$ is an arbitrary closed quantum subgroup. We use here Tannakian duality, and more specifically, the following result:

\begin{proposition}
Given an inclusion $G\subset S_N^+$, with the corresponding fundamental corepresentations denoted $u\to w$, we have the following formula:
$$C(G)=C(S_N^+)\Big/\Big(T\in Hom(u^{\otimes k},u^{\otimes l}),\forall k,l\in\mathbb N,\forall T\in Hom(w^{\otimes k},w^{\otimes l})\Big)$$
with the Hom-spaces at left being taken in a formal sense.
\end{proposition}

\begin{proof}
We recall that for a corepresentation $v=(v_{ij})$, the condition $T\in Hom(v^{\otimes k},v^{\otimes l})$ means that we have $Tv^{\otimes k}=v^{\otimes l}T$, the tensor powers being given by $v^{\otimes r}=(v_{i_1\ldots i_r,j_1\ldots j_r})$. We can formally use these notions for any square matrix over any $C^*$-algebra, and in particular, for the fundamental corepresentation of $C(S_N^+)$. Thus, the collection of relations $T\in Hom(u^{\otimes k},u^{\otimes l})$, one for each choice of an intertwiner $T\in Hom(w^{\otimes k},w^{\otimes l})$, produce a certain ideal of $C(S_N^+)$, and so the quotient algebra in the statement is well-defined.

This latter algebra is isomorphic to $C(G)$, due to Woronowicz's Tannakian results in \cite{wo2}. For a short, recent proof here, using basic Hopf algebra theory, see \cite{mal}.
\end{proof}

Now back to our matrix model questions, given $G\subset S_N^+$, the idea is that of constructing a universal model space $X_G\subset X_N$, by using the above Tannakian equations:

\begin{theorem}
Given a closed subgroup $G\subset S_N^+$, there is a universal flat model
$$\pi_G:C(G)\to M_N(C(X_G))\quad:\quad w_{ij}\to[u\to u_{ij}]$$
with $X_G\subset X_N$ being obtained by using the Tannakian relations for $G$.
\end{theorem}

\begin{proof}
This follows by using Proposition 6.1. Indeed, in order to construct the universal flat model, we need a universal solution to the following problem:
$$\begin{matrix}
C(S_N^+)&\to&M_N(C(X_N))\\
\\
\downarrow&&\downarrow\\
\\
C(G)&\to&M_N(C(X_G))
\end{matrix}$$

According to Proposition 6.1, the solution to this latter question is given by the following construction, with the Hom-spaces at left being taken as usual, in a formal sense:
$$C(X_G)=C(X_N)\Big/\Big(T\in Hom(u^{\otimes k},u^{\otimes l}),\forall k,l\in\mathbb N,\forall T\in Hom(w^{\otimes k},w^{\otimes l})\Big)$$

To be more precise here, the quotient algebra on the right is well-defined, by Gelfand duality this algebra must be of the form $C(X_G)$, for a certain algebraic submanifold $X_G\subset X_N$, and this manifold has by definition the desired universality property.
\end{proof}

As a basic example, for the quantum groups $G=S_N,S_N^+$ we recover the spaces $X_N^\circ,X_N$ from Definition 5.4, and the representations $\pi_N^\circ,\pi_N$. In general, the above construction remains of course quite theoretical. Observe for instance that if $G\subset S_N^+$ has the property that $u_{ij}=0$ for a certain pair $(i,j)$, then $X_G\subset X_N$ will collapse to the null space.

Our construction, while definitely still in need of some improvements, provides however a framework for an extension of the conjectures in \cite{bne}. We have here:

\begin{conjecture}
Assuming that $G\subset S_N^+$ has suitable uniformity properties:
\begin{enumerate}
\item The universal flat representation $\pi_G$ is inner faithful.

\item If the dual quantum group $\widehat{G}$ is amenable, $\pi_G$ is stationary.
\end{enumerate}
\end{conjecture}

Here it is of course not very clear what ``uniformity'' should mean, but, according to Theorem 5.5 above, the usual symmetric group $S_N$ should be definitely included. In general, we should look here for a property which implies $u_{ij}=0$, for any $i,j$. The difficulty comes from the fact that many interesting examples of quantum permutation groups, such as the group duals, do have this unwanted property, $u_{ij}=0$ for some $i,j$. 

It is not clear either how to formulate something more precise regarding the measure on $X_G$, which is needed at the stationarity statement. Observe however that Theorem 5.5 suggests that $X_G$ might be an homogeneous space, times a discrete space. One question here, which is open for some time already, is that of explicitely computing $X_4$.

As a conclusion now, the world of compact quantum groups seems to be very related to the world of random matrices. Among the concrete questions raised by the present work, perhaps the most important is that of enlarging the stationary model framework, as to cover the deformed Fourier models from \cite{ban}, \cite{bb2}, \cite{bic}. Of particular interest here is the question on how these deformed Fourier models appear, and this is related to some delicate questions in real algebraic geometry, around the notion of defect \cite{tzy}. In fact, understanding the algebraic geometry of the various matrix models themselves, and its relation with the associated quantum groups, looks like a quite interesting problem.

\end{document}